\documentclass[a4paper,10pt,reqno]{amsart}

\usepackage{amsmath,graphics}
\usepackage{amssymb}
\usepackage{amsfonts}
\usepackage{latexsym}
\usepackage{eucal}
\usepackage{mathrsfs}
\usepackage[dvips]{graphicx}

\theoremstyle{plain}
\newtheorem{thm}{Theorem}[section]
\newtheorem{propo}[thm]{Proposition}
\newtheorem{lem}[thm]{Lemma}

\theoremstyle{definition}

\renewcommand{\Re}{{\rm Re}}
\renewcommand{\Im}{{\rm Im}}
\newcommand{\R}{\mathbb{R}}
\newcommand{\C}{\mathbb{C}}
\newcommand{\Z}{\mathbb{Z}}
\newcommand{\N}{\mathbb{N}}

\newcommand{\hh}{\mathscr{H}}
\renewcommand{\H}{\mathbb{H}^2}

\title[]{Determinants of Laplacians on random hyperbolic surfaces}

\author[F.~Naud]{By Fr\'ed\'eric Naud}
\address{%
Fr\'ed\'eric Naud\\
Institut Math\'ematique de Jussieu\\
Universit\'e Pierre et Marie Curie, 4 place Jussieu, 75252 Paris Cedex 05\\
France.
}
\email{frederic.naud@imj-prg.fr}



\begin{document}
\bibliographystyle{plain}

\maketitle

\begin{abstract}
For sequences $(X_j)$ of random closed hyperbolic surfaces with volume $\mathrm{Vol}(X_j)$ tending to infinity, we prove that  
there exists a universal constant $E>0$  such that for all $\epsilon>0$, the regularized determinant of the Laplacian satisfies 
$$\frac{\log \det(\Delta_{X_j})}{\mathrm{Vol}(X_j)}\in [E-\epsilon,E+\epsilon]$$
with high probability as $j\rightarrow +\infty$.
This result holds for various models of random surfaces, including the Weil--Petersson model.
\end{abstract}

\section{Introduction and results}
\subsection{On determinants}
Let $X=\Gamma \backslash \H$ be a compact connected hyperbolic surface obtained as a quotient of the hyperbolic plane $\H$ by a discrete co-compact torsion-free group of orientation-preserving isometries. The hyperbolic Laplacian $\Delta_X$ on $L^2(X)$ has
a pure point spectrum which we denote by 
$$0=\lambda_0<\lambda_1\leq \ldots \leq \lambda_j\leq\ldots. $$
For all $s\in \C$ with $\Re(s)$ large enough, we know by Weyl's law that the spectral zeta function
$$\zeta_X(s)=\sum_{j=1}^\infty \frac{1}{\lambda_j^s}$$
is well defined and holomorphic.  The regularized determinant is then usually defined by
$$\log\det(\Delta_X):=-\zeta_X'(0),$$
provided one can prove an analytic extension to $s=0$ of $\zeta_X$. Practically, one performs a meromorphic continuation by noticing that for large $\Re(s)$ we have
$$\zeta_X(s)=\frac{1}{\Gamma(s)} \int_0^\infty t^{s-1} (\mathrm{Tr}(e^{-t\Delta_X})-1)dt, $$
where $e^{-t\Delta_X}$ is the heat semi-group. The small time asymptotics at $t\sim 0$ of the heat kernel is the main tool which allows to ``renormalize" the divergent behavior at $t=0$ and obtain the meromorphic continuation, see for example Chavel \cite[Page 156]{Chavel}.

In the literature, Polyakov's string theory \cite{Polyakov, DP2} has emphasized the role of determinants on Riemann surfaces. In particular, the computation of ``partition functions" in perturbative string theory involves formal sums over all genera of averages of determinants over the moduli space which have proved since then to be divergent, see Wolpert \cite{Wolpert}. Several authors have provided \cite{Sarnak,BS,DP} some explicit formulas for regularized determinants for various Laplace-like operators on Riemann surfaces.  In small genus, it is possible to compute accurately such determinants by reducing to certain sums over closed geodesics which provide a fast convergence, see \cite{PR,Str}. In variable curvature, the behavior of determinants in a conformal class has been studied
by Osgood, Phillips and Sarnak \cite{OPS}, in particular the constant curvature metric maximizes the determinant.

In higher dimensions, determinants of Laplacians on differential forms are related to the so-called analytic torsion, which in turn is related to important topological invariants by results of Cheeger and M\"uller \cite{Cheeger,Muller}. In particular the work of Bergeron--Venkatesh \cite{BV} establishes exponential growth of the analytic torsion for certain families of covers of arithmetic manifolds.

In the case of Riemann surfaces, if $\Gamma$ is a co-compact arithmetic Fuchsian group derived from a quaternion algebra, one can define congruence covers
$$X_{\mathscr{P}}:=\Gamma(\mathscr{P})\backslash \H$$
 of $X=\Gamma \backslash \H$ by looking at {\it prime ideals} $\mathscr{P}$ in the ring of integers of the corresponding number field. 
We denote by $\Vert \mathscr{P} \Vert$ the norm of ideals.
Using the uniform spectral gap of these surfaces proved by Sarnak and Xue in \cite{SX}, together with the fact that the injectivity radius goes to infinity as $\Vert \mathscr{P} \Vert\rightarrow \infty$, see in \cite{KSV}, one can readily show (for example by using the arguments from \cite{B1} or as a direct application of Theorem \ref{main2}) that
$$\lim_{\Vert \mathscr{P} \Vert\rightarrow \infty} \frac{\log\det \Delta_{X_{\mathscr{P}}}}{\mathrm{Vol}(X_{\mathscr{P}})}=E,$$
where $E>0$ is some universal constant. Arithmetic surfaces being highly non-generic, it is therefore natural to ask if this behavior is typical among larger families of surfaces whose volume (equivalently genus) goes to infinity.

\subsection{Models of random surfaces}
In this paper we will focus on the behaviour of determinants of the Laplacian in the large volume (equivalently large genus) regime, using probabilistic tools. The first historical model of random compact Riemann surfaces in the mathematics literature is perhaps the model of {\bf Brooks--Makover \cite{BM}} which is based on random $3$-regular graphs as follows. Consider $\mathcal{G}_n$ a $3$-regular graph on $2n$ vertices, endowed with an orientation \footnote{an orientation on a graph is a function
which assigns to each vertex $v$ of the graph a cyclic ordering of the edges emanating
from $v$.} $\mathcal{O}$. On the (finite) set of all possible pairs 
$(\mathcal{G}_n, \mathcal{O})$, one can put a probability measure (which is not the uniform measure) introduced first by Bollob\'as \cite{Boll1}, see \cite[Section 5]{BM}
for a summary on this construction which allows tractable computations in the large $n$ regime.
By glueing $2n$ ideal hyperbolic triangles according to $(\mathcal{G}_n, \mathcal{O})$ in such a way that the feet of the altitudes in adjacent triangles match up, 
one then obtains a random finite area hyperbolic surface
$S_n^O:=S^O(\mathcal{G}_n, \mathcal{O})$ with $\mathrm{Vol}(S_n)=2\pi n$. It is possible to show, see \cite{BM2},  that all surfaces in $S_n^O$ are actually (non ramified) covers of the modular
surface $\mathrm{PSL}_2(\Z) \backslash \H$.

 One can then conformally compactify $S_n^O$ by cutting cusps and filling them with discs. Provided that $S_n^O$ has genus at least $2$, we then denote by $S_n^C$ the unique hyperbolic surface in the conformal class of this compactification, see \cite[Section 3]{BM}. In $\S 4$ of the same paper they also show that there exists a constant $C_0>0$ such that with high probability as $n\rightarrow +\infty$, 
$$\mathrm{Vol}(S_n^C)\geq C_0 n.$$
Most of the geometric properties of $S_n^O$ (and then $S_n^C$, after a mild loss) can be read off from the combinatorics of $\mathcal{G}_n$. 

Another more recent discrete model of random surfaces is the so-called {\bf random cover model} which has been studied and used recently in \cite{MP,MNP,N1}.
In what follows, we fix a compact surface
$X=\Gamma \backslash \H$, "the base surface". 
Let $\phi_n:\Gamma \rightarrow \mathcal{S}_n$ be a group homomorphism, where $\mathcal{S}_n$ is  the symmetric group of permutations of $[n]:=\{1,\ldots,n\}$. The discrete group $\Gamma$ acts on $\H\times [n]$ by
$$\gamma.(z,j):=(\gamma(z), \phi_n(\gamma)(j)). $$
The resulting quotient $X_n:=\Gamma \backslash (\H\times [n])$ is then a finite cover of degree $n$ of $X$, possibly not connected. By considering the (finite) space of all homomorphism $\phi_n:\Gamma\rightarrow \mathcal{S}_n$, endowed with the uniform probability measure, we obtain a notion of random covering surfaces of degree $n$, $X_n\rightarrow X$. 
Let us remark that we can also view (up to isometry) the random cover $X_n$ as 
$$X_n=\bigsqcup_{k=1}^p \Gamma_k\backslash \H,$$
where each $\Gamma_k$ is the (a priori non-normal) subgroup of $\Gamma$ given by 
$$\Gamma_k=\mathrm{Stab}_\Gamma(i_k)=\{ \gamma \in \Gamma\ :\ \phi_n(\gamma)(i_k)=i_k \},$$ where $i_1,\ldots,i_p \in [n]$ are representatives of the orbits of $\Gamma$ (acting on $[n]$ via $\phi_n$). In general, the cover $X_n$ is not connected, but it follows directly from \cite{LS1} that the probability that this cover is connected  tends to $1$ as $n$ goes to infinity. In this model, we have $\mathrm{Vol}(X_n)=n\mathrm{Vol}(X)$.

A smooth model of random hyperbolic surfaces is given by the {\bf moduli space} $\mathscr{M}_g$ of closed hyperbolic surfaces
with genus $g$, up to isometry.  It is often defined as the quotient
$$\mathscr{M}_g=\mathscr{T}_g/\mathrm{MCG}, $$
where $\mathscr{T}_g$ is the Teichm\"uller space of hyperbolic metrics on a surface $S$ of genus $g$ and 
$$\mathrm{MCG}=\mathrm{Diff}(S)/\mathrm{Diff}_0(S)$$
is the group of isotopy classes of diffeomorphisms on $S$, aka the mapping class group. We refer the reader for example to \cite{Buser}, chapter 6 for more details. A symplectic form $\omega_{WP}$
lives naturally on $\mathscr{T}_g$ and descends to the moduli space, endowing it with a natural notion of volume, Weil--Petersson volume.
The moduli space is a non-compact finite dimensional orbifold, but as a consequence of Bers' theorem on pants decomposition, see \cite[Theorem 5.1.2]{Buser} it has a finite volume with respect to this Weil--Petersson volume. We can therefore normalize this measure and obtain a probability measure on
$\mathscr{M}_g$. Notice that in this case if $X\in \mathscr{M}_g$, $\mathrm{Vol}(X)=4\pi(g-1)$ by Gauss--Bonnet. The calculation of Weil--Petersson volumes of 
the moduli space by Mirzakhani \cite{Mirza} has made possible \cite{Mir1} the large genus asymptotic analysis of various geometric and spectral quantities, see for example \cite{Monk1, WX,WL} for recent works in that direction.

For all of the previous models, we will denote by ${\mathbb P}$ the associated probability measure, which depends either on $n$ or $g$, which are both proportional to the volume. We say that an event $\mathcal{A}$ is asymptoticaly almost sure (a.a.s.), or holds with high probability, if $\mathbb{P}(\mathcal{A})$ tends to $1$ as the volume of surfaces tends to infinity. The expectation of any relevant random variable will also be denoted by $\mathbb{E}$.

\subsection{Main result}
\begin{thm}
\label{main1}
There exists a universal constant $E>0$ such that for all the above models of random surfaces, for all $\epsilon>0$ we have
$$\frac{\log\det(\Delta_{X})}{\mathrm{Vol}(X)} \in [E-\epsilon,E+\epsilon],$$
a.a.s as $\mathrm{Vol}(X)\rightarrow +\infty.$
\end{thm}
The constant $E$ is actually explicit and is approximately $0.0538$, see $\S2$ for an exact description. This result shows that exponential growth of the determinant is typical when the genus goes to infinity. This low dimensional result is consistent, in a much simpler setting, 
with the conjectures on the exponential growth of the analytic torsion and the torsion homology for higher dimensional hyperbolic manifolds, see for example
the paper of Bergeron--Venkatesh \cite{BV} and references therein. 

Remark that the above statement says that the random variable $ \frac{\log\det(\Delta_{X})}{\mathrm{Vol}(X)}$ converges in probability to the constant $E$. What about other modes of convergence? 
It is possible to derive from Theorem \ref{main1} a convergence result
for the expectation of $\vert \log \det(\Delta)\vert^\beta$, see section $\S 5$, Theorem \ref{limit1}: for both models of random covers and Weil--Petersson, we show the existence of exponents $\beta>0$ such that
$$\lim_{\mathrm{Vol}(X)\rightarrow \infty} \mathbb{E}\left(\frac{\vert \log \det(\Delta_{X})\vert^\beta}{\mathrm{Vol}(X)^\beta} \right)=E^\beta.$$ 

The paper is organized as follows. In $\S 2$ we recall how one establishes an identity for $\log\det(\Delta_X)$ which involves infinite sums over closed geodesics via the Heat trace formula.
In $\S 3$ we prove an abstract Theorem which guarantees the exponential growth of determinants as long as a certain natural list of assumptions are satisfied. 
 These hypotheses turn out to be valid a.a.s. for the probabilistic models listed above, and this is established in $\S 4$. In $\S 5$, we derive from Theorem \ref{main1} a convergence result
 for the expectation, based on some moments estimates for the systole and the smallest positive eigenvalue.
 
 \noindent {\bf Acknowledgement.} It is a pleasure to thank my neighbor Bram Petri for several discussions around this work. Thanks to Yuhao Xue and an anonymous referee for pointing out an improvement of Theorem \ref{main2}. I also thank Zeev Rudnick for his reading and comments. Finally, Yunhui Wu and Yuxin He have recently
 shown to me that Theorem \ref{limit1} can also be refined in the Weil--Petterson case, see in $\S 5$ for details.

\section{Heat kernels and determinants}
In this section, we recall some standard calculations on regularized determinants mostly taken from \cite[Appendix B]{BS}. Our goal is to show how the heat trace formula allows
to derive an identity for $\log \det \Delta_X$.

On the hyperbolic plane $\H$, the heat kernel $p_t(x,y)$ (see for example \cite{Buser} chapter 7) has an explicit formula given by
$$p_t(x,y)=\frac{\sqrt{2}e^{-t/4}}{(4\pi t)^{3/2}}\int_{d(x,y)}^\infty \frac{re^{-r^2/4t}dr}{\sqrt{\cosh r-\cosh d(x,y)}},$$
where $d(x,y)$ denotes the hyperbolic distance in $\H$. On the quotient $X=\Gamma \backslash \H$, we can recover the heat kernel by summing over the group, i.e.,
$$h_t^X(x,y)=\sum_{\gamma \in \Gamma} p_t(x,\gamma y).$$
Convergence of the above series on any compact subset of $\H$ is guaranteed by the lattice counting bound
\begin{equation}
\label{Lattice}
N_\Gamma(x,y,T):=\#\{ \gamma \in \Gamma \ :\ d(\gamma x,y)\leq T\}=O(e^T),
\end{equation}
which is standard and follows from a basic volume argument.
The semi-group of operators $e^{-t\Delta_X}$ is then of trace class and one has the explicit ``heat trace formula" 
\begin{equation}
\label{heatrace1}
\mathrm{Tr}(e^{-t\Delta_X})=\sum_j e^{-t \lambda_j}=\mathrm{Vol(X)} \frac{e^{-t/4}}{(4\pi t)^{3/2}}\int_0^\infty \frac{re^{-r^2/4t}}{\sinh(r/2)}dr \\
\end{equation}
$$
+\frac{e^{-t/4}}{(4\pi t)^{1/2}}\sum_{k\geq 1}\sum_{\gamma \in \mathcal{P}} \frac{\ell(\gamma)}{2\sinh(k\ell(\gamma)/2)}e^{-(k\ell(\gamma))^2/4t},$$
where $\mathcal{P}$ stands for the set of primitive conjugacy classes in $\Gamma$ (i.e. oriented primitive closed geodesics on $X$) and if $\gamma \in \mathcal{P}$,
$\ell(\gamma)$ is the length. For more details on the calculation of this trace and more generally Selberg's formula, see \cite{Hejhal1,Buser}. Setting 
$$S_X(t):= \frac{e^{-t/4}}{(4\pi t)^{1/2}}\sum_{k\geq 1}\sum_{\gamma \in \mathcal{P}} \frac{\ell(\gamma)}{2\sinh(k\ell(\gamma)/2)}e^{-(k\ell(\gamma))^2/4t},$$
it is easy to see from the spectral side of the trace formula that $\vert S_X(t)-1\vert$ is exponentially small as $t\rightarrow +\infty$. We also observe that $S_X(t)$ is exponentially small as $t\rightarrow 0$. We now write (for $\Re(s)$ large)
$$\zeta_X(s)=\frac{1}{\Gamma(s)} \int_0^\infty t^{s-1} (\mathrm{Tr}(e^{-t\Delta_X})-1)dt= \zeta^{(1)}_X(s)+\zeta^{(2)}_X(s),$$
where we have set
$$\zeta^{(1)}_X(s)=:\mathrm{Vol(X)}\frac{1}{4\pi \Gamma(s)} \int_0^\infty t^{s-1} \frac{e^{-t/4}}{\sqrt{4\pi} t^{3/2}}\int_0^\infty \frac{re^{-r^2/4t}}{\sinh(r/2)}drdt, $$
and
$$\zeta^{(2)}_X(s)=:\frac{1}{\Gamma(s)} \int_0^\infty t^{s-1} (S_X(t)-1)dt.$$
Writing (for $\Re(s)$ large)
$$\zeta^{(2)}_X(s)=\frac{1}{\Gamma(s)} \int_0^1 t^{s-1} (S_X(t)-1)dt+\frac{1}{\Gamma(s)} \int_1^\infty t^{s-1} (S_X(t)-1)dt$$
$$=\frac{-1}{\Gamma(s+1)}+\frac{1}{\Gamma(s)} \int_0^1 t^{s-1}S_X(t)dt+\frac{1}{\Gamma(s)} \int_1^\infty t^{s-1} (S_X(t)-1)dt,$$
we notice that the last two integrals make sense for all $s\in \C$.
Therefore $\zeta^{(2)}_X(s)$ has an analytic extension to $\C$ and using elementary facts on the gamma function (in particular that it has a simple pole at $s=0$ with residue $1$), 
we have that 
$$-\left. \frac{d}{ds}\right \vert_{s=0} \zeta^{(2)}_X(s)=-\Gamma'(1)-\int_0^1 \frac{S_X(t)}{t}dt-\int_1^\infty \frac{(S_X(t)-1)}{t}dt.$$
On the other hand, for large $\Re(s)$ we use that \footnote{for example one can use the identity valid for $x\in \R$
$$\tanh(\pi x)=\frac{1}{\pi} \sum_{k=0}^\infty \frac{2x}{x^2+(k+1/2)^2}=\frac{1}{\pi}\int_0^\infty\frac{\sin(ux)}{\sinh(u/2)}du.$$}
$$\frac{e^{-t/4}}{\sqrt{4\pi} t^{3/2}}\int_0^\infty \frac{re^{-r^2/4t}}{\sinh(r/2)}dr=\int_{-\infty}^{+\infty} x\tanh(\pi x)e^{-(x^2+1/4)t}dx, $$
see \cite[Page 593]{BS}, and one can compute the Mellin transform to obtain (again for $\Re(s)$ large)
$$\zeta^{(1)}_X(s)=\frac{2\mathrm{Vol}(X)}{4\pi} \int_0^\infty \frac{u\tanh(\pi u)}{(u^2+1/4)^s}du.$$
This function can be analytically continued to $s=0$, see for example \cite[Appendix B]{BS}, and the value can be actually computed as
$$-\left. \frac{d}{ds}\right \vert_{s=0} \zeta^{(1)}_X(s)=\frac{\mathrm{Vol}(X)}{4\pi}\left ( 4\zeta'(-1)-1/2+\log(2\pi)\right):=\mathrm{Vol(X)}E,$$
with $\zeta'(-1)=1/12-\log(A)$ and $A$ is the so-called {\it Glaisher--Kinkelin} constant, which is for example defined by
$$A=\lim_{n\rightarrow \infty} \frac{\prod_{k=1}^n k^k}{e^{-n^2/4}n^{n^2/2+n/2+1/12}}.$$
We have $E\approx 0,0538$. Using in addition that $\Gamma'(1)=-\gamma_0$, where 
$$\gamma_0=\lim_{n\rightarrow +\infty} \sum_{k=1}^n \frac{1}{k}-\log(n),$$
is the {\it Euler constant}, we have obtained the celebrated identity
\begin{equation}
\label{detformula}
\log\det \Delta_X=\mathrm{Vol}(X)E+\gamma_0-\int_0^1 \frac{S_X(t)}{t}dt-\int_1^\infty \frac{(S_X(t)-1)}{t}dt.
\end{equation}
This formula can be interpreted multiplicatively via Selberg zeta function at $s=1$, see \cite{Sarnak, BS, DP}.
\section{An abstract deterministic statement}
Theorem \ref{main1} actually follows from a more general deterministic result for sequences of compact surfaces satisfying certain hypotheses denoted by $\hh_1$ and $\hh_2$.
More precisely, if $(X_k)$ is a sequence of compact connected hyperbolic surfaces with $\mathrm{Vol}(X_k)\rightarrow +\infty$, let $\mathcal{P}_k$ denote the set of oriented primitive closed geodesics on $X_k$, and let
$\Delta_k$ be the hyperbolic Laplacian on $X_k$. We also denote by $\ell_0(X_k)$ the length of the shortest closed geodesic on $X_k$. Let $C>0$, $\eta>0$, $L>0$ and $0<\alpha<1/2$  be some constants.

We say that the sequence $(X_k)$ satisfies hypothesis $\hh_1(\eta)$ if for all $k \in \N$ we have 
\begin{equation}
\label{hyp1}
\lambda_1(\Delta_k)\geq \eta.
\end{equation}
We say that the sequence $(X_k)$ satisfies hypothesis $\hh_2(C,L,\alpha)$ if for all $k \in \N$ we have the following
 bound on the {\it number of closed geodesics}:  
 \begin{equation}
 \label{hyp2}
N_k(L):=N_{X_k}(L):=\#\{ (\gamma,m) \in \mathcal{P}_k\times \N\ \ : \  m\ell(\gamma)\leq L \} \leq  C\mathrm{Vol}(X_k)^{\alpha}.
\end{equation}
In this paper $\N=\{1,2,\ldots\}$ is the set of natural integers starting at $1$. We point out that exponential growth of Laplace determinants is established in the literature for families of covers for which a uniform spectral gap holds and the injectivity radius of the manifolds goes to infinity, see for example \cite{BV} and \cite{B1}. Typical examples are congruence covers of arithmetic hyperbolic manifolds and Laplacians twisted by a "strongly acyclic" representation which ensures a uniform spectral gap. 
In random models of surfaces, having the injectivity radius grow to infinity is {\it atypical} and we establish the result under the weaker assumption
of small growth of the number of closed geodesics with bounded length.

Theorem \ref{main1} will follow from the following deterministic result.
\begin{thm}
\label{main2}
Fix some $\eta>0$ and $0<\alpha<1$. Assume that $(X_k)$ satisfies $\hh_1(\eta)$ and $\hh_2(C_0,L_0,\alpha)$ with $L_0=2 \mathrm{arcsinh}(1)$ for some $C_0>0$.
Then for all $\epsilon>0$, there exists $L_\epsilon>0$  such that if in addition $(X_k)$ satisfies $\hh_2(C_\epsilon,L_\epsilon,\alpha)$ for some $C_\epsilon>0$,  then uniformly for all $\mathrm{Vol}(X_k)$ large,
$$\frac{\log\det(\Delta_{X_k})}{\mathrm{Vol}(X_k)} \in [E-\epsilon,E+\epsilon],$$
where $E>0$ is the universal constant from above.
\end{thm}
All the positive constants denoted by $C_1,C_2,\ldots,C_j$ below depend only on $C_0,L_0$. Before we give a proof of Theorem \ref{main2}, we will need a preliminary Lemma which is needed to control uniformly sums over closed geodesics.

\begin{lem}
\label{count1}
Under hypothesis $\hh_2(C_0,L_0,\alpha)$ where 
$L_0=2 \mathrm{arcsinh}(1),$
 there exists $C_1>0$ such that for all $k$ and all $T\geq 0$,
$$N_k(T)\leq C_1 \mathrm{Vol}(X_k)e^{T}.$$
\end{lem}
\begin{proof} A result of Buser (\cite{Buser} Lemma 6.6.4) says that for any compact connected hyperbolic surface of genus $g$, the number
of oriented closed geodesics with length $\leq T$ which are {\it not iterates of primitive closed geodesics of length $\leq 2 \mathrm{arcsinh}(1)$} is bounded from above by
$$(g-1)e^{T+6}.$$
Therefore we have
$$N_k(T)\leq \frac{\mathrm{Vol}(X_k)}{4\pi}e^{T+6}
+ \#\{ (\gamma,m) \in \mathcal{P}_k\times \N\ \ : \  m\ell(\gamma)\leq T \ \mathrm{and}\ \ell(\gamma)\leq 2 \mathrm{arcsinh}(1)\}.$$
On the other hand we have
$$\#\{ (\gamma,m) \in \mathcal{P}_k\times \N\ \ : \  m\ell(\gamma)\leq T \ \mathrm{and}\ \ell(\gamma)\leq 2 \mathrm{arcsinh}(1)\}
\leq \sum_{\ell(\gamma)\leq \mathrm{arcsinh}(1)} \frac{T}{\ell(\gamma)}.$$
We can observe that by definition of $N_{k}(L)$, we have
$$N_{k}(L)=\sum_{m\ell(\gamma)\leq L} 1=\sum_{\ell(\gamma)\leq L}  \left [ \frac{L}{\ell(\gamma)}  \right],$$
where $[.]$ is the integer part. Hence we can write
\begin{equation}
\label{small_sums}
\sum_{\ell(\gamma)\leq L} \frac{1}{\ell(\gamma)} \leq \frac{2}{L} N_{X_k}(L). 
\end{equation}
Going back to the estimate of $N_{X_k}(T)$ and using $\hh_2(C_0,L_0,\alpha)$ with $L_0=\mathrm{arcsinh}(1)$, we get
$$N_k(T)\leq \mathrm{Vol}(X_k)e^T\frac{e^6}{4\pi}+\frac{2}{\mathrm{arcsinh}(1)}N_{X_k}(\mathrm{arcsinh}(1))$$
$$\leq C_1\mathrm{Vol}(X_k)e^T$$
the proof is done. 
\end{proof}

\begin{lem}
\label{keylemma}
Under hypothesis $\hh_1(\eta)$ and $\hh_2(C_0,L_0,\alpha)$ where $L_0$ is as above, there exists  $C_2>0$ such that for all $k$ and all $t\geq 1$,
$$\left \vert  S_{X_k}(t)-1\right\vert\leq C_2 \mathrm{Vol}(X_k)e^{-\eta_0 t},$$
where $\eta_0=\min(\eta,1/4)$.
\end{lem}
\begin{proof} In this proof we will use Vinogradov's notation $A\ll B$ meaning that $A\leq C B$ where $C>0$ is a {\it universal constant}.
By formula (\ref{heatrace1}) we have
$$\mathrm{Tr}(e^{-t\Delta_{X_k}})-1=\mathrm{Vol(X_k)} \frac{e^{-t/4}}{(4\pi t)^{3/2}}\int_0^\infty \frac{re^{-r^2/4t}}{\sinh(r/2)}dr+S_{X_k}(t)-1,$$
and therefore we get 
$$\vert S_{X_k}(t)-1 \vert \ll \mathrm{Vol}(X_k)e^{-t/4}+\sum_{j=1}^\infty e^{-t\lambda_j(X_k)},$$
for all $t\geq 1$. On the other hand using the uniform spectral gap we have,
$$\sum_{j=1}^\infty e^{-t\lambda_j(X_k)}=\sum_{j=1}^\infty e^{-(t-1)\lambda_j-\lambda_j}\leq e^{-(t-1)\eta}\mathrm{tr}(e^{-\Delta_{X_k}}).$$
Going back to formula (\ref{heatrace1}) with $t=1$, we have also
$$ \mathrm{tr}(e^{-\Delta_{X_k}})\ll \mathrm{Vol}(X_k)+S_{X_k}(1),$$
 and
$$S_{X_k}(1)\ll  \sum_{m\geq 1}\sum_{\gamma \in \mathcal{P}_k} \frac{\ell(\gamma)}{2\sinh(m\ell(\gamma)/2)}e^{-(m\ell(\gamma))^2/4}$$
$$\ll \int_0^\infty \frac{u}{2\sinh(u/2)}e^{-u^2/4}dN_k(u),$$
where we have used the Stieltjes integral notation with the measure $dN_k$ associated to the counting function $N_k$.
We can use Lemma \ref{count1} to bound $N_k(u)$ as
$$N_k(u)\leq C_1 \mathrm{Vol}(X_k)e^{u},$$
and a summation by parts shows that 
$$\int_0^\infty \frac{u}{2\sinh(u/2)}e^{-u^2/4}dN_k(u)=-\int_0^\infty N_k(u) \frac{d}{du}\left \{   \frac{u}{2\sinh(u/2)} e^{-u^2/4}\right\}du $$
$$\ll C_1\mathrm{Vol}(X_k),$$
which ends the proof. 
\end{proof}

\begin{proof}[Proof of Theorem \ref{main2}]
First notice that by formula (\ref{detformula}), we have 
$$\left \vert \frac{\log\det(\Delta_{X_k})}{\mathrm{Vol}(X_k)}-E\right \vert \leq O( \mathrm{Vol}(X_k)^{-1})+ \mathcal{D}^{(1)}_{X_k}+\mathcal{D}^{(2)}_{X_k},$$
where
$$\mathcal{D}^{(1)}_{X_k}= \frac{1}{\mathrm{Vol}(X_k)}\int_1^\infty \frac{\vert S_{X_k}(t)-1\vert}{t}dt,\ \mathcal{D}^{(2)}_{X_k}=\frac{1}{\mathrm{Vol}(X_k)}\int_0^1 \frac{S_{X_k}(t)}{t}dt.$$
Let us fix $\epsilon>0$. We first investigate $\mathcal{D}^{(1)}_{X_k}$. Using $\hh_1(\eta)$ and $\hh_2(C_0,\alpha,L_0)$, we can use Lemma \ref{keylemma} and write
$$\int_1^\infty \frac{\vert S_{X_k}(t)-1\vert}{t}dt\leq \int_1^R\frac{S_{X_k}(t)}{t}dt+\log(R)+C_2 \mathrm{Vol}(X_k)\int_R^\infty\frac{e^{-\eta_0 t}}{t}dt,$$
for any $R>1$.
Fixing $R=R(\epsilon)$ so large that
$$C_2 \int_R^\infty\frac{e^{-\eta_0 t}}{t}dt\leq \epsilon,$$
we have
$$\mathcal{D}^{(1)}_{X_k}\leq \frac{1}{\mathrm{Vol}(X_k)} \int_1^R\frac{S_{X_k}(t)}{t}dt+  \frac{\log(R)}{\mathrm{Vol}(X_k)}+\epsilon.$$
We now pick $L_1>1$ (to be adjusted later on) and write
$$ \int_1^R\frac{S_{X_k}(t)}{t}dt=\int_1^R\frac{S_{X_k}^{L_1,-}(t)}{t}dt +\int_1^R\frac{S_{X_k}^{L_1,+}(t)}{t}dt,$$
where 
$$S_{X_k}^{L_1,-}(t)=\frac{e^{-t/4}}{(4\pi t)^{1/2}}\sum_{m\ell(\gamma)\leq L_1} \frac{\ell(\gamma)}{2\sinh(m\ell(\gamma)/2)}e^{-(m\ell(\gamma))^2/4t},\ $$
$$S_{X_k}^{L_1,+}(t)=\frac{e^{-t/4}}{(4\pi t)^{1/2}}\sum_{m\ell(\gamma)> L_1} \frac{\ell(\gamma)}{2\sinh(m\ell(\gamma)/2)}e^{-(m\ell(\gamma))^2/4t}.$$
Clearly we have
$$\int_1^R\frac{S_{X_k}^{L_1,+}(t)}{t}dt\leq C_3 \sum_{m\ell(\gamma)> L_1} \frac{\ell(\gamma)}{2\sinh(m\ell(\gamma)/2)}e^{-(m\ell(\gamma))^2/4R},$$
for some universal constant $C_3>0$. Using Lemma \ref{count1} and a summation by parts, we deduce that
$$\int_1^R\frac{S_{X_k}^{L_1,+}(t)}{t}dt\leq C_4 \mathrm{Vol(X_k)} \int_{L_1}^\infty \left \vert \frac{d}{du} \left ( \frac{u}{\sinh(u/2)} e^{-u^2/(4R)}\right) \right\vert e^u du.$$
We now take $L_1=L_1(\epsilon)$ so large that 
$$C_4 \int_{L_1}^\infty \left \vert \frac{d}{du} \left ( \frac{u}{\sinh(u/2)} e^{-u^2/(4R)}\right) \right\vert e^u du<\epsilon.$$
We now observe that if $\hh_2(C,L_1,\alpha)$ holds, we have
$$\int_1^R\frac{S_{X_k}^{L,-}(t)}{t}dt=\sum_{m\ell(\gamma)\leq L_1} \frac{\ell(\gamma)}{2\sinh(m\ell(\gamma)/2)}\int_1^R e^{-(m\ell(\gamma))^2/4t} \frac{e^{-t/4}}{\sqrt{4\pi}t^{3/2}}dt$$
$$\leq C_5 L_1 N_k(L_1)\leq C_5CL_1\mathrm{Vol}(X_k)^{\alpha}\leq C_6(L_1,C)\mathrm{Vol}(X_k)^{\alpha},$$
for some possibly large constant $C_6(L_1,C)>0$. In a nutshell, we have obtained, provided that $\hh_2(C,L_1,\alpha)$ is satisfied with $L_1=L_1(\epsilon)$ taken large enough,
$$\limsup_{\mathrm{Vol}(X_k)\rightarrow +\infty} \mathcal{D}^{(1)}_{X_k}\leq 2\epsilon.$$
We now turn our attention to $\mathcal{D}^{(2)}_{X_k}$, and this is where a good control of sums over short geodesics is required. We first use the same idea as above by writing
$$\mathcal{D}^{(2)}_{X_k}=\frac{1}{\mathrm{Vol}(X_k)}\int_0^1 \frac{S_{X_k}(t)}{t}dt= \frac{1}{\mathrm{Vol}(X_k)}\int_0^1\frac{S_{X_k}^{L_2,-}(t)}{t}dt +
\frac{1}{\mathrm{Vol}(X_k)}\int_0^1\frac{S_{X_k}^{L_2,+}(t)}{t}dt.$$
Writing for $t>0$,
$$S_{X_k}^{L_2,+}(t)\leq C_7 t ^{-1/2} \sum_{m\ell(\gamma)> L_2} \frac{\ell(\gamma)}{2\sinh(m\ell(\gamma)/2)}e^{-(m\ell(\gamma))^2/4t},$$
where $C_7>0$ is universal, we have by Fubini
$$\int_0^1\frac{S_{X_k}^{L_2,+}(t)}{t}dt\leq C_7 \sum_{m\ell(\gamma)> L_2} \frac{\ell(\gamma)}{2\sinh(m\ell(\gamma)/2)}G(m\ell(\gamma)),$$
where for $u>0$,
$$G(u)=\int_0^1 t^{-3/2}e^{-u^2/4t}dt.$$
Notice that $u\mapsto G(u)$ is a decreasing function and by a change of variable we have actually for all $u>0$,
$$G(u)=\frac{4}{u}\int_{u/2}^\infty e^{-x^2}dx.$$
We have therefore the bound
$$G(u)=\frac{4}{u}\int_{u/2}^\infty e^{-x^2/2-x^2/2}dx\leq \frac{4}{u}e^{-u^2/8}\int_0^\infty e^{-x^2/2}dx=\frac{2\sqrt{2\pi}}{u}e^{-u^2/8}.$$
As a consequence we get for $L_2>1$
$$\int_0^1\frac{S_{X_k}^{L_2,+}(t)}{t}dt\leq C_8 \sum_{m\ell(\gamma)> L_2} \frac{\ell(\gamma)}{2\sinh(m\ell(\gamma)/2)} e^{-(m\ell(\gamma))^2/8},$$
and by using Lemma \ref{count1} and a summation by parts, we can definitely fix $L_2=L_2(\epsilon)$ large enough so that
$$\int_0^1\frac{S_{X_k}^{L_2,+}(t)}{t}dt\leq C_8 \sum_{m\ell(\gamma)> L_2} \frac{\ell(\gamma)}{2\sinh(m\ell(\gamma)/2)} e^{-(m\ell(\gamma))^2/8}\leq \mathrm{Vol}(X_k)\epsilon. $$

From the above bound on $G(u)$ we also deduce 
$$ \int_0^1\frac{S_{X_k}^{L_2,-}(t)}{t}dt\leq C_9\sum_{m\ell(\gamma)\leq  L_2} \frac{\ell(\gamma)}{\sinh(m\ell(\gamma)/2)}\frac{1}{m\ell(\gamma)}$$
$$\leq C_9'  \sum_{m\ell(\gamma)\leq  L_2} \frac{1}{m^2\ell(\gamma)}.$$
By writing
$$ \sum_{m\ell(\gamma)\leq  L_2} \frac{1}{m^2\ell(\gamma)}=\sum_{m=1}^\infty \frac{1}{m^2} \sum_{\ell(\gamma)\leq L_2/m} \frac{1}{\ell(\gamma)}
\leq \frac{\pi^2}{6} \sum_{\ell(\gamma)\leq L_2} \frac{1}{\ell(\gamma)},$$
we can use estimate (\ref{small_sums}) and $\hh_2(C,L_2(\epsilon),\alpha)$ and we have again as above
$$\int_0^1\frac{S_{X_k}^{L_2,-}(t)}{t}dt\leq C_{10}(L_2,C) \mathrm{Vol}(X_k)^{\alpha},$$
where $C_{10}(L_2,C)>0$ is some (possibly very large) constant depending on $L_2,C$ and $0<\alpha<1$. We have therefore shown, that whenever $(X_k)$ satisfies $\hh_2(C,L(\epsilon),\alpha)$
for some $C>0$ and with $L(\epsilon)=\max\{L_1,L_2 \}$ we have
$$\limsup_{\mathrm{Vol}(X_k)\rightarrow +\infty}  \left \vert \frac{\log\det(\Delta_{X_k})}{\mathrm{Vol}(X_k)}-E\right \vert\leq 3\epsilon.$$
Theorem \ref{main2} is proved. 
\end{proof}

\section{Hypotheses $\hh_1$ and $\hh_2$ hold with high probability}
Theorem \ref{main1} follows immediately from Theorem \ref{main2} if one can establish for the three models of random hyperbolic surfaces considered here that there exists $\eta>0$ and such that $\hh_1(\eta)$ holds a.a.s. and also that there exists $0<\alpha<1$ such that for all $L$ large, one can find $C>0$ such that $\hh_2(C,L,\alpha)$ also holds  a.a.s.

 Indeed we then have for all $\epsilon>0$, and all $\mathrm{Vol}(X)$ large enough
$$\mathbb{P}\left (     \frac{\log\det(\Delta_{X})}{\mathrm{Vol}(X)} \in [E-\epsilon,E+\epsilon] \right)\geq $$
$$\mathbb{P}\left( X\in \hh_1(\eta)\cap \hh_2(C_0,L_0,\alpha)\cap \hh_2(C,L(\epsilon),\alpha) \right),$$
with
$$\lim_{\mathrm{Vol}(X)\rightarrow \infty}\mathbb{P}\left( X\in \hh_1(\eta)\cap \hh_2(C_0,L_0,\alpha)\cap \hh_2(C,L(\epsilon),\alpha)\right)=1.$$

\subsection{Random covers}
First we recall that the cover $X_n\rightarrow X$ may not be connected but we know from \cite{LS1} that
$$\lim_{n\rightarrow +\infty} \mathbb{P}(X_n\ \mathrm{connected})=1.$$
We can therefore either restrict ourselves to connected surfaces $X_n$ or modify the definition of the regularized determinant by setting
$$\zeta_{X_n}(s)=\frac{1}{\Gamma(s)} \int_0^\infty t^{s-1} (\mathrm{Tr}(e^{-t\Delta_X})-d_n)dt,$$
where $d_n$ is the number of connected component of $X_n$ and use the fact that $d_n=1$ with high probability. 

It was recently shown in \cite{MNP} that there exists a uniform spectral gap for random covers $X_n$ a.a.s. as $n\rightarrow +\infty$. More precisely
we have for all $0<\eta<\min\{3/16,\lambda_1(X)\}$, 
$$\lim_{n\rightarrow+\infty} \mathbb{P}(\lambda_1(X_n)\geq \eta)=1. $$
Therefore  $\hh_1(\eta)$ holds a.a.s. provided $\eta$ is taken small enough. On the other hand property $\hh_2$ is less obvious from the existing litterature 
and will require some explanations. We recall that given a random homomorphism $\phi_n:\Gamma\rightarrow \mathcal{S}_n$, one can define
a unitary representation $\rho_n$ of $\Gamma$ by setting
$$\rho_n(\gamma)(f):=f\circ\phi_n(\gamma)^{-1},$$
where $f\in L^2([n])$, and the representation space is $L^2([n])\simeq \C^n$. The main interest of this representation is the following fact, often called ``Venkov--Zograf induction formula". For all $\Re(s)>1$, one can define
the Selberg zeta function of $X_n$ by
$$Z_{X_n}(s):=\prod_{m\geq 0} \prod_{\gamma \in \mathcal{P}_{X_n}}\left ( 1-e^{-(s+m)\ell(\gamma)}\right).$$
One can also look at the twisted Selberg zeta function of the base $X=\Gamma \backslash \H$ defined for $\Re(s)>1$ by
$$Z_{X,\rho_n}(s):=\prod_{m\geq 0} \prod_{\gamma \in \mathcal{P}_{X}}\det\left ( I-\rho_n(\gamma)e^{-(s+m)\ell(\gamma)}\right).$$
It turns out that we have for all $s$, $Z_{X_n}(s)=Z_{X,\rho_n}(s)$, see \cite[Page 51]{Venkov}. By computing logarithmic derivatives we have for all $\Re(s)>1$,
$$\frac{Z'_{X_n}(s)}{Z_{X_n}(s)}=\sum_{\gamma \in \mathcal{P}_{X_n}}\sum_{q\geq 1} \frac{\ell(\gamma)e^{-sq\ell(\gamma)}}{1-e^{-q\ell(\gamma)}}=\sum_{\gamma \in \mathcal{P}_{X}}\sum_{q\geq 1} \frac{\ell(\gamma)\mathrm{tr}(\rho_n(\gamma^q))e^{-sq\ell(\gamma)}}{1-e^{-q\ell(\gamma)}}.$$
Let $\phi \in C_0^\infty(\R^+)$ be a compactly supported smooth test function, and set 
$$\psi(s):=\int_0^\infty e^{sx}\phi(x)dx.$$
One can check that $\psi(s)$ is actually analytic on $\C$ and by Fourier inversion formula we have for all $A>1$,
$$\frac{1}{2i\pi}\int_{A-i\infty}^{A+i\infty} \frac{Z'_{X_n}(s)}{Z_{X_n}(s)}\psi(s)ds=
\sum_{\gamma \in \mathcal{P}_{X_n}}\sum_{q\geq 1} \frac{\ell(\gamma)}{1-e^{-q\ell(\gamma)}}\phi(q\ell(\gamma)) $$
$$=\sum_{\gamma \in \mathcal{P}_{X}}\sum_{q\geq 1} \frac{\ell(\gamma)\mathrm{tr}(\rho_n(\gamma^q))}{1-e^{-q\ell(\gamma)}}\phi(q\ell(\gamma)).$$
See for example \cite{JN} $\S 3$ for more details on the derivation of this formula.
By carefully choosing the test function $\phi$ we deduce that for all $\mathcal{L}\in \R^+$, we recover the identity
$$\sum_{\gamma \in \mathcal{P}_{X_n}}\sum_{q\geq 1 \atop q\ell(\gamma)=\mathcal{L}} \ell(\gamma)=\sum_{\gamma \in \mathcal{P}_{X}}\sum_{q\geq 1\atop q\ell(\gamma)=\mathcal{L}} \ell(\gamma)\mathrm{tr}(\rho_n(\gamma^q)).$$
Notice that this formula can be proved directly by group theoretic arguments, see for example in \cite{PF1}, in the proof of theorem 7.1.
From this identity we deduce that for all $L>0$, we have
$$\sum_{\gamma \in \mathcal{P}_{X_n}}\sum_{q\geq 1 \atop q\ell(\gamma)\leq L} \ell(\gamma)=\sum_{\gamma \in \mathcal{P}_{X}}\sum_{q\geq 1\atop q\ell(\gamma)\leq L} \ell(\gamma)\mathrm{tr}(\rho_n(\gamma^q)).$$
In particular we have
$$\ell_0(X_n)N_{X_n}(L)\leq  L \sum_{\gamma \in \mathcal{P}_{X}}\sum_{q\geq 1\atop q\ell(\gamma)\leq L}\mathrm{tr}(\rho_n(\gamma^q)),$$
where $\ell_0(X_n)$ denotes the shortest closed geodesic length on $X_n$.
We point out that we have actually $\mathrm{tr}(\rho_n(\gamma^q))=\mathrm{Fix}(\phi_n(\gamma^q))$, where $\mathrm{Fix}(\sigma)$ denotes the number of fixed points of the
permutation $\sigma$ acting on $[n]$. From the combinatorial analysis of Magee-Puder \cite{MP1,MNP}, we know that for all primitive $\gamma \in \Gamma$
and $q\geq 1$, we have
$$\lim_{n\rightarrow \infty}  \mathbb{E}(\mathrm{Fix}(\phi_n(\gamma^q)))=d(q),$$
where $d(q)$ stands for the number of divisors of $q$. Noticing that in the random cover model, 
we have always $\ell_0(X_n)\geq \ell_0(X)$, this his is enough to conclude that for all $L$, we have
$$ \lim_{n\rightarrow \infty}  \mathbb{E}(N_{X_n}(L))\leq C(\Gamma,L),$$
where $C(\Gamma,L)>0$ is some (possibly large) constant. Applying Markov's inequality, we get that for all $\varepsilon>0$ and $L$ fixed, 
$$\lim_{n\rightarrow \infty} \mathbb{P}\left(N_{X_n}(L)\leq \mathrm{Vol}(X_n)^\varepsilon\right)=1.$$
As a conclusion, in the random cover model, $\hh_2(C,L,\alpha)$ is satisfied a.a.s. for all $L$ large and all $\alpha>0$.

\subsection{Brooks--Makover model} In the paper \cite[Theorem 2.2]{BM}, they show that there exists a constant $C_1>0$ such that 
as $n\rightarrow \infty$,
$$\mathbb{P}(\lambda_1(S_n^C)\geq C_1)\rightarrow 1,$$
In other words, property $\hh_1(\eta)$ is satisfied a.a.s. for some $\eta>0$.

We point out that contrary to the model of random covers, the systole of $S_n^C$ can be arbitrarily small, but we actually know that there exists $C_2>0$ such that as $n\rightarrow +\infty$, $$\mathbb{P}(\ell_0(S_n^C)\geq C_2)\rightarrow 1.$$
Counting results for closed geodesics
follow from the later work of Petri \cite{Petri1}.
More precisely, one can derive from \cite{Petri1} the following fact.
\begin{propo}
\label{bm}
For all $L>0$ fixed, we can find an integer $N_L$ and a finite set of words
$$\mathcal{W}_L\subset \{l,r\}^{N_L},$$
such that with high probability as $n\rightarrow \infty$,
$$N_{n}(L):=\#\{(\gamma,m) \in \mathcal{P}_{S_n^C}\times \N^*\ :\ m\ell(\gamma)\leq L \} \leq \sum_{w \in \mathcal{W}_L} Z_{n,w},$$
where $Z_{n,w}$ are integer-valued random variables.
In addition each $Z_{n,w}$ converges in the sense of moments (and hence in distribution) as $n\rightarrow \infty$ to a Poisson variable with expectation $\lambda_w>0$.
\end{propo}
By applying Markov's inequality, one deduces readily that for all $\varepsilon>0$ we have a.a.s. 
$$\sum_{w \in \mathcal{W}_L} Z_{n,w}\leq n^\epsilon\leq C_L\mathrm{Vol}(S_n^C)^\varepsilon.$$
This is enough to conclude that for all $\varepsilon>0$, with high probability  as $n\rightarrow \infty$, we have 
$$N_n(L)\leq C\mathrm{Vol}(S_n^C)^\varepsilon,$$
and therefore property $\hh_2(C,\alpha,L)$ holds for any choice of $\alpha>0$, just like in the previous model.

Let us now give some details on the proof of Proposition \ref{bm}. The first step is to reduce the problem to a counting bound for $S_n^O$. In \cite[Section 3]{BM}, they introduce the notion of ``large cusps" condition for $S_n^O$. This condition is satisfied a.a.s for $S_n^O$ as $n\rightarrow \infty$ see \cite[Theorem 2.1]{BM}. The main interest of this condition is  \cite[Theorem 3.2]{BM}, see also \cite[Lemma 2.5]{Petri1}, which allows to show that provided this ``large cusps" condition is satisfied, one can bound
$$N_n(L)\leq \#\{(\gamma,m) \in \mathcal{P}_{S_n^0}\times \N^*\ :\ m\ell(\gamma)\leq 2L \}=:N_n^O(2L).$$
As explained in \cite{BM}, $\S 4$, closed geodesics and their length in $S_n^O$ can be described via the combinatorial data of $(\mathcal{G}_n,\mathcal{O})$: any closed geodesic in 
$S_n^O$ corresponds to a word $w\in \{l,r\}^N$, for some $N>0$. To this word one can associate a matrix $M_w$ in $SL_2(\N)$ via the rule
$$M_w=W_1\ldots W_N,$$
where $W_j=\mathcal{L}$ if $w_j=l$ and $W_j=\mathcal{R}$ if $w_j=r$, where
$$\mathcal{L}=\left ( \begin{array}{cc} 1&1\\0&1 \end{array}\right),\  \mathcal{R}=\left ( \begin{array}{cc} 1&0\\1&1 \end{array}\right).$$
The length $\ell_w$ of the corresponding geodesic on $S_n^O$ is then given by 
$$\mathrm{tr}(M_w)=2\cosh(\ell_w/2).$$
All we need to check is that fixing $L$ implies finiteness of the corresponding set of words (by bounding their word length).
This is done in \cite[Lemma 3.1]{PW}. One obtains therefore that $L$ being fixed, there exists a finite subset 
$\mathcal{W}_L\subset \{l,r\}^{N_L}$ for some large $N_L>0$, such that for all $n$, all closed geodesics with length $\leq 2L$ on $S_n^O$ are given by words $w\in W_L$. Proposition
\ref{bm} now follows directly from \cite[Theorem B]{Petri1}.
\subsection{Weil--Petersson model} 
If $f$ is a measurable non-negative function on moduli space $\mathscr{M}_g$, we will denote by
$\int_{\mathscr{M}_g} f(X)dX$ the corresponding integral with respect to Weil--Petersson volume. In the latter $V_g$ will denote the Weil--Petersson volume of 
$\mathscr{M}_g$ so that the expectation of $f$ is given by
$$\mathbb{E}_g(f):= \frac{1}{V_g}\int_{\mathscr{M}_g} f(X)dX.$$ 
In \cite{Mir1}, the following fact was proved. There exists $\eta>0$ such that as $g\rightarrow +\infty$, we have
$$\mathbb{P}( \lambda_1(X)\geq \eta)\rightarrow 1.$$
The constant $\eta$ given by Mirzakhani follows from Cheeger's inequality and an estimate a.a.s. of Cheeger's isoperimetric constant. It was shown independently
in \cite{WX,WL} that one can actually take $\eta=3/16-\epsilon$, and more recently $\eta=2/9-\epsilon$ by Anantharaman and Monk \cite{AM}. This shows that $\hh_1(\eta)$ holds with high probability as $g\rightarrow +\infty$ 
for some universal $\eta>0$. In \cite[Theorem 4.2]{Mir1} and the remark after, Mirzakhani proved that there exists a universal $\epsilon_0>0$ such that for all $0<\epsilon\leq \epsilon_0$, one has for all $g$ large $$ \mathbb{P}( \ell_0(X)\leq \epsilon)\leq C \epsilon^2,$$
where $C$ is uniform in $g$. In particular for all $\varepsilon>0$, we get for all $g$ large
$$\mathbb{P}\left( \ell_0(X)\geq \mathrm{Vol}(X)^{-\varepsilon}\right)\geq 1-O(\mathrm{Vol}(X)^{-2\varepsilon}).$$
On the other hand, in the paper \cite{MP}, Mirzakhani and Petri showed that for all $L>0$ fixed, the random variable
$$N_g^0(L):=\#\{\gamma \in \mathcal{P}_X\ :\ \ell(\gamma)\leq L \}$$
converges in distribution as $g\rightarrow +\infty$ to a Poisson variable $Z_{\lambda_L}$ with parameter
$$\lambda_L:=\int_0^L\frac{e^t+e^{-t}-2}{2t}dt.$$
Moreover, we have also convergence of all moments with $p\in \N$
$$\lim_{g\rightarrow \infty} \mathbb{E}(\left (N_g^0(L)\right)^p)=\mathbb{E}(Z_{\lambda_L}^p).$$
An application of Markov's inequality then shows that for all $\varepsilon>0$, for all $g$ large enough
$$\mathbb{P}(N_g^0(L)\leq \mathrm{Vol}(X)^\varepsilon)\geq 1-C_L\mathrm{Vol}(X)^{-\varepsilon}.$$
To control the counting function
$$N_g(L):=\#\{(\gamma,m) \in \mathcal{P}_X\times \N^*\ :\ m\ell(\gamma)\leq L \},$$
we write
$$N_g(L)\leq \sum_{m=1}^{[L/\ell_0(X)]+1}N_g^0(L/m)\leq N_g^0(L)\left(1+\frac{L}{\ell_0(X)}\right).$$
For all $\varepsilon>0$, we have with high probability as $g\rightarrow +\infty$
$$N_g(L)\leq \mathrm{Vol}(X)^\varepsilon+L\mathrm{Vol}(X)^{2\varepsilon}=O_L( \mathrm{Vol}(X)^{2\varepsilon}),$$
and thus for all $L$ large, there exists $C_L>0$ such that
$\hh_2(C_L,L,\alpha)$ holds with high probability for all $\alpha>0$ on $\mathscr{M}_g$ as $g\rightarrow +\infty$. 
\section{Convergence results for moments of $\log\det(\Delta_X)$.}
In this last section, we give the proof of the following fact.
\begin{thm}
\label{limit1}
 In the Weil--Petersson model, for all $0<\beta<1$ we have
$$\lim_{g\rightarrow \infty} \frac{1}{V_g (4\pi(g-1))^\beta}\int_{\mathscr{M}_g} \left \vert \log \det(\Delta_X)\right \vert^\beta dX =E^\beta.$$
In the random cover model, Let $\chi_0=\mathbf{1}_{\{X_n\ \mathrm{connected}\}}$. Then as $n\rightarrow +\infty$, for all $\beta>0$, we have 
$$\lim_{n\rightarrow \infty} \mathbb{E}\left(\frac{\vert \log \det(\Delta_{X_n})\vert^\beta}{\mathrm{Vol}(X_n)^\beta}\chi_0 \right)=E^\beta.$$
\end{thm}
\begin{proof}[Proof for the Weil--Petterson model]
The proof for the Weil--Petersson model is a rather direct consequence of Theorem \ref{main1} and some estimates of Mirzakhani \cite{Mirza}. We first need an a priori estimate for
$\vert \log \det(\Delta_X)\vert$ which follows from similar ideas as in Theorem \ref{main2}, without the probabilistic input. We use Vinogradov's notation $\ll$ where the implied constant
is universal.

By using Buser's counting bound \cite{Buser}, as in the proof of Lemma \ref{count1}, we have the following universal bound for the number of closed geodesics of a surface with genus $g$:
$$N_X(L)\leq (g-1)e^{L+6}+\frac{L}{\ell_0(X)}(3g-3)$$
\begin{equation}
\label{univ1}
\ll \mathrm{Vol}(X)e^L\left ( 1+\frac{1}{\ell_0(X)}\right),
\end{equation}
for some universal constant $A_1>0$. We now prove an a priori upper bound for $\vert\log \det(\Delta_X)\vert$. Following ideas of Wolpert \cite{Wolpert}, it is convenient to write for all $\Re(s)$ large,
$$\zeta_X(s)=\sum_{0<\lambda_j < 1/4} \lambda_j^{-s}+\frac{1}{\Gamma(s)}\int_{0}^\infty t^{s-1}\left ( \mathrm{tr}(e^{-t\Delta_X})-\sum_{0\leq \lambda_j<1/4}e^{-t\lambda_j}  \right)dt,$$
which can be rewritten as
$$\zeta_X(s)=\sum_{0<\lambda_j < 1/4} \lambda_j^{-s}-\sum_{0\leq \lambda_j<1/4}H(s,\lambda_j)+ \zeta^{(1)}_X(s)+$$
$$\frac{1}{\Gamma(s)}\int_{1}^\infty t^{s-1}\left ( S_X(t)-\sum_{0\leq \lambda_j<1/4}e^{-t\lambda_j}  \right)dt+\frac{1}{\Gamma(s)}\int_{0}^1 t^{s-1}S_X(t)dt,$$
where we have set for $\Re(s)$ large,
$$H(s,\lambda):= \frac{1}{\Gamma(s)}\int_{0}^1 t^{s-1}e^{-\lambda t}dt.$$
By integration by parts and elementary properties of the Euler gamma function, we then observe that for all $\lambda\in [0,1/4]$, $s\mapsto H(s,\lambda)$ has an analytic extension to $s=0$.
Morevover, if we set 
$$C(\lambda):=\left. \frac{d}{ds}\right \vert_{s=0}H(\lambda,s),$$
then there exists a universal constant $A_1>0$ such that for all $\lambda \in [0,1/4]$, $$\vert C(\lambda)\vert\leq A_1.$$
This formula leads to the identity
$$\log\det(\Delta_X)= \sum_{0<\lambda_j < 1/4} \log(\lambda_j)+\sum_{0<\lambda_j < 1/4} C(\lambda_j)- \left. \frac{d}{ds}\right \vert_{s=0}\zeta^{(1)}_X(s)$$
$$-\int_0^1 \frac{S_X(t)}{t}dt-\int_1^\infty \frac{(S_X(t)-\sum_{0\leq \lambda_j<1/4} e^{-\lambda_jt })}{t}dt. $$
By mimicking the proof of Lemma \ref{keylemma} and using the above counting bound, we deduce that for all $t\geq 1$,
$$\left \vert S_X(t)-\sum_{0\leq \lambda_j<1/4}e^{-t\lambda_j}  \right\vert\ll \mathrm{Vol}(X)\left ( 1+\frac{1}{\ell_0(X)}\right)e^{-t/4}.$$
By Fubini and the estimate on $u\mapsto G(u)$ we have also
$$\int_0^1 \frac{S_X(t)}{t}dt\ll \sum_{m,\gamma} \frac{e^{-(m\ell(\gamma))^2/8}}{m\ell(\gamma)} 
\ll \sum_{m\ell(\gamma)\leq 1} \frac{1}{m\ell(\gamma)}+\sum_{m\ell(\gamma)>1} e^{-(m\ell(\gamma))^2/8}$$
$$\ll  \sum_{m\ell(\gamma)\leq 1} \frac{1}{m\ell(\gamma)}+ \mathrm{Vol}(X)\left(1+\frac{1}{\ell_0(X)}\right).$$
By noticing that we have
$$ \sum_{m\ell(\gamma)\leq 1} \frac{1}{m\ell(\gamma)}\ll \frac{\log^+ \ell_0^{-1}(X)}{\ell_0(X)}N_X^0(1),$$
where $\log^+(x)=\max\{0,\log(x)\}$ and $N_X^0(L)$ is the counting function for primitive closed geodesics, we have obtained the estimate:
$$\vert \log \det(\Delta_X)\vert \ll\mathrm{Vol}(X)\left ( 1+\vert \log(\lambda_*(X))\vert  +\frac{1}{\ell_0(X)}+
\right)+\frac{\log^+ \ell_0^{-1}(X)}{\ell_0(X)}N_X^0(1) ,$$
where we have set $\lambda_*(X)=\min \{\lambda_1(X),1/4 \}$, and we have used the rough bound of Buser \cite{Buser}:
$$\#\{ \lambda_j < 1/4\}\leq 4g-3=1+\frac{\mathrm{Vol}(X)}{\pi}.$$
Notice that the optimal bound of Otal--Rosas \cite{OR}
$$\#\{ \lambda_j < 1/4\}\leq 2g-2,$$
won't make any difference here.

This estimate of $\vert \log \det(\Delta_X)\vert$ is consistent with the fact that $\det(\Delta_X)$ has exponential growth when $X$ approaches certain boundary points of the (compactified) moduli space, a fact that was rigourously established by Wolpert in \cite{Wolpert}. In particular the so-called ``bosonic Polyakov integral" involving the determinant over the moduli space is indeed infinite, see \cite{Wolpert}.

Using the inequality for all $\beta>0$ and all $a_j\geq 0$,
$$\left ( \sum_{j=1}^4 a_j \right)^\beta\leq 4^\beta (\max_j a_j)^\beta \leq 4^\beta \left ( \sum_{j=1}^4 a_j^\beta\right),$$
we end up with the estimate 
$$\frac{\vert \log \det(\Delta_X)\vert^\beta}{\mathrm{Vol}(X)^\beta} \ll \left (1+ \vert \log(\lambda_*(X))\vert^\beta+\frac{1}{\ell_0(X)^\beta} +
\mathrm{Vol}(X)^{-\beta} \left (\frac{\log^+ \ell_0^{-1}(X)}{\ell_0(X)}N_X^0(1) \right)^\beta\right).$$
From Mirzakhani \cite[Corollary 4.3]{Mirza}, we know that 
$$\int_{\mathscr{M}_g} \frac{1}{\ell_0(X)}dX\leq C V_g, $$
where $C>0$ is independent of $g$, from which we can deduce easily that for all $0<\alpha<1$, we have
$$ \int_{\mathscr{M}_g} \left (\frac{\log^+ \ell_0^{-1}(X)}{\ell_0(X)}\right)^\alpha dX\leq C V_g,$$
for some $C>0$ {\it uniform in $g$}. Indeed, for all $0<\alpha<1$, there exists a universal consant $r_0>0$ such that
$0<x\leq r_0$ implies
$$\left (\frac{\vert \log x\vert}{x}\right)^\alpha\leq \frac{1}{x}.$$
By integrating over $\{X \in \mathscr{M}_g\ :\ \ell_0(X)\leq r_0\}$ and using Mirzakhani's bound we get the desired bound while on the complementary set
$$X\mapsto \left (\frac{\log^+ \ell_0^{-1}(X)}{\ell_0(X)}\right)^\alpha $$
is uniformly bounded.

On the other hand, Cheeger's inequality says that
$$\lambda_1(X)\geq \frac{h_X^2}{4},$$
where $h_X$ is the so-called Cheeger constant of $X$, which is defined by an isoperimetric quantity, see for example in \cite{Buser}, chapter 8. 
Again By Mirzakhani \cite[Theorem 4.8]{Mirza}, we know that for all $0\leq \alpha<2$, we have
$$\int_{\mathscr{M}_g} \frac{1}{(h_X)^\alpha}dX\leq C V_g ,$$
where $C>0$ is again uniform with respect to $g$. In particular we deduce that for all $\beta>0$,
$$\int_{\mathscr{M}_g} \vert \log(\lambda_*(X))\vert^\beta dX\leq C V_g.$$
Assuming that $\beta<1$, we choose $p>1$ such that $\beta<p\beta<1$ and let $q$ be such that $1/p+1/q=1$. By H\"older's inequality we get
$$\mathbb{E}_g \left(\left (\frac{\log^+ \ell_0^{-1}(X)}{\ell_0(X)}N_X^0(1) \right)^\beta \right) \leq 
\left [\mathbb{E}_g \left(\left (\frac{\log^+ \ell_0^{-1}(X)}{\ell_0(X)}\right)^{\beta p}\right)\right]^{1/p} \left [ \mathbb{E}_g \left (N_X^0(1)^{\beta q}   \right)\right]^{1/q},$$
which by the convergence of moments in Mirzakhani-Petri \cite{MP} and the above remarks is uniformly bounded as $g\rightarrow +\infty$.
We therefore have shown that for all $\beta<1$, there exists $C>0$ independent of $g$ such that
\begin{equation}
\label{moment1}
\int_{\mathscr{M}_g} \frac{\vert \log \det(\Delta_X)\vert^\beta}{\mathrm{Vol}(X)^\beta} dX\leq C V_g.
\end{equation}
We now fix $\epsilon>0$, and $0<\beta<1$. By Theorem \ref{main1}, there exists a subset $\mathcal{A}_g(\epsilon)\subset \mathscr{M}_g$, with $\mathbb{P}(\mathcal{A}_g(\epsilon))\rightarrow 1$ as $g\rightarrow +\infty$,
such that for all $X \in \mathcal{A}_g(\epsilon)$, 
$$ (E-\epsilon)^\beta\leq \frac{\vert \log \det(\Delta_X)\vert^\beta}{\mathrm{Vol}(X)^\beta} \leq (E+\epsilon)^\beta.$$
Therefore we have
$$ \frac{1}{V_g}\int_{\mathscr{M}_g} \frac{\vert \log \det(\Delta_X)\vert^\beta}{\mathrm{Vol}(X)^\beta} dX\leq (E+\epsilon)^\beta+
\frac{1}{V_g}\int_{\mathcal{A}_g(\epsilon)^c} \frac{\vert \log \det(\Delta_X)\vert^\beta}{\mathrm{Vol}(X)^\beta} dX,$$
while
$$(E-\epsilon)^\beta\mathbb{P}(\mathcal{A}_g(\epsilon))\leq \frac{1}{V_g}\int_{\mathscr{M}_g} \frac{\vert \log \det(\Delta_X)\vert^\beta}{\mathrm{Vol}(X)^\beta} dX.$$
We now apply H\"older's inequality and estimate (\ref{moment1}). Since $0<\beta<1$, let $q>1$ be chosen such that  $0<q\beta <1$ and let $0<p<\infty$ be such that $1/p+1/q=1$, we have
$$\frac{1}{V_g}\int_{\mathcal{A}_g(\epsilon)^c} \frac{\vert \log \det(\Delta_X)\vert^\beta}{\mathrm{Vol}(X)^\beta} dX 
\leq \left (\mathbb{P}(\mathcal{A}_g(\epsilon)^c)\right)^{1/p} \left ( \frac{1}{V_g}\int_{\mathscr{M}_g} \frac{\vert \log \det(\Delta_X)\vert^{q\beta}}{\mathrm{Vol}(X)^{q\beta}} dX \right)^{1/q}$$
$$\leq C  \left(\mathbb{P}(\mathcal{A}_g(\epsilon)^c\right )^{1/p}.$$
Since $\lim_{g\rightarrow +\infty} \mathbb{P}(\mathcal{A}_g(\epsilon)^c)=0$, we definitely have for all $g$ large enough
$$(E-\epsilon)^\beta-\epsilon\leq \frac{1}{V_g}\int_{\mathscr{M}_g} \frac{\vert \log \det(\Delta_X)\vert^\beta}{\mathrm{Vol}(X)^\beta} dX\leq (E+\epsilon)^\beta +\epsilon,$$
and the proof is done for the Weil--Petterson case. 
\end{proof}

We would like to mention that in this smooth Weil--Petterson model, it is very likely that by using  finer estimates one can improve the result to the following statement: for any $\beta\in (0,2)$ we have
$$\lim_{g\rightarrow \infty} \frac{1}{V_g (4\pi(g-1))^\beta}\int_{\mathscr{M}_g} \left \vert \log \det(\Delta_X)\right \vert^\beta dX =E^\beta,$$
while for all $\beta\geq 2$, 
$$\int_{\mathscr{M}_g} \vert \log \det(\Delta_X)\vert^\beta dX=+\infty.$$
This fact was pointed out to the author by {\it Yunhui Wu and Yuxin He} in a private communication and will be published in a separate paper.

\begin{proof}[Proof for the random covers model] We recall that we have
$\chi_0=\mathbf{1}_{\{X_n\ \mathrm{connected}\}}$. In this model, the systole is bounded uniformly from below, so the a priori bound for 
$\log \det(\Delta_{X_n})$ whenever $X_n$ is connected, is actually
$$\frac{\vert \log \det(\Delta_{X_n})\vert^\beta}{\mathrm{Vol}(X_n)^\beta} \leq A \left (1+\vert \log(\lambda_*(X))\vert^\beta \right),$$
for some constant $A>0$ independent of $n$. Using the same arguments as above based on H\"older's inequality, the result follows directly from the next fact.
\begin{propo} Assuming that $X_n$ is connected, then we have for all $n$ large,
$$\lambda_1(X_n)\geq \frac{C_\Gamma}{n^{3/2}},$$
where $C_\Gamma$ depends only on the base surface $X=\Gamma \backslash \H$.
Consequently for all exponent $\beta$ with $0<\beta$, 
$$\mathbb{E}\left( \vert \log(\lambda_*(X))\vert^\beta \chi_0 \right) \leq C,$$
where $C>0$ is uniform with respect to $n$.
\end{propo}
We first need to prove a deterministic lower bound on $\lambda_1(X_n)$, provided that $X_n$ is connected. We know that the spectrum of
$\Delta_{X_n}$ coincides, with multiplicity, with the spectrum of $\Delta_{\rho_n}$, which is the Laplacian on the base surface twisted by the unitary representation $\rho_n$ of $\Gamma$ defined previously. See for example \cite[Page 51]{Venkov}. If $X_n$ is connected, then 
$$\lambda_1(X_n)=\min\{\lambda_1(X) ,\lambda_0(\Delta_{\rho_n^0})\},$$
where $\rho_n^0$ is the representation given by 
$$\rho_n^0(\gamma)U:=U_{\phi_n^{-1}(\gamma)},$$
where $U \in V_n^0:=\{ U\in \C^n\ :\ \sum_{j=1}^nU(j)=0\}$ and $U_{\phi_n^{-1}(\gamma)}(j)=U(\phi_n^{-1}(\gamma)(j))$. Notice that $\rho_n^0$ is just the orthogonal complement to the trivial representation
in $\rho_n$. Let us fix a system of generators of $\Gamma$, denoted by $S$. A result of Sunada \cite{Sunada} then says that there exists $C_S$ depending only on $X$ and $S$ such that
$$\lambda_0(\Delta_{\rho_n^0})\geq C_S \inf_{U\in V_n^0 \atop \Vert U\Vert=1} \max_{\gamma \in S} \Vert \rho_n^0(\gamma) U-U\Vert. $$
Let us take $U\in V_n^0$ such that $\Vert U\Vert=1$. We therefore have
$$1=\Vert U \Vert^2 \leq n \max_{j} (\Re(U(j))^2+\Im(U(j))^2),$$
and we can assume without loss of generality that we have $\Re(U(j_0))\geq \frac{1}{\sqrt{2n}}$ for some $j_0\in [n]$. Because we have in addition 
$$\sum_j  \Re(U(j))=0,$$
there exists also $j_1 \in [n]$ such that $\Re(U(j_1))\leq 0$. If $X_n$ is connected, then $\Gamma$ acts transitively on $[n]$ via $\phi_n$ and there exists $\gamma_0 \in \Gamma$ such that $\phi_n(\gamma_0)^{-1}(j_0)=j_1$. To bound the word length of $\gamma_0$, consider the graph with set of vertices $[n]$ and define edges by connecting $i$ to $j$ if there exists 
$g\in S$ such that $\phi_n(g)(i)=j$. By transitivity of the action, this graph is connected and thus has diameter less than $n-1$. Therefore we can choose $\gamma_0$ with word length (with respect to $S$) less than $n-1$.
We now have
$$\frac{1}{\sqrt{2n}}\leq \vert \Re(U(j_0))-\Re(U(j_1))\vert \leq  \Vert \rho_n^0(\gamma_0) U-U\Vert.$$
Writing
$$\gamma_0=g_1g_2\ldots g_m,$$
with $g_j \in S$ and $m\leq n-1$, we have therefore
$$\Vert \rho_n^0(\gamma_0) U-U\Vert\leq \sum_{j=1}^m \Vert \rho_n^0(g_j) U -U\Vert\leq (n-1) \max_{g\in S}\Vert \rho_n^0(g) U-U\Vert,$$
which yields
$$ \max_{g\in S}\Vert \rho_n^0(g) U-U\Vert \geq \frac{1}{(n-1)\sqrt{2n}}.$$
The first claim of the proposition is proved. Alternatively, one can use directly a result of Brooks \cite{Brooks} which relates $\lambda_1(X_n)$ to the Cheeger constants of Schreier graphs of the covers (with a choice of generators of $\Gamma$) to obtain a similar lower bound $\lambda_1(X_n)\geq C_\Gamma n^{-2}$ which is slightly worse but good enough for our purpose.

From the proof of the uniform spectral gap in \cite{MNP} one can directly derive that for all $1/4>r>0$, for all $\epsilon>0$, as $n\rightarrow \infty$,
$$\mathbb{P}(\lambda_1(X_n)\leq r)\leq \frac{C_\epsilon}{n^{4\sqrt{(1/4-r)}-1-\epsilon}}.$$
From that we deduce that for all $\alpha<1$ we can find $r_\alpha>0$ such that 
$$\mathbb{P}(\lambda_1(X_n)\leq r_\alpha)\leq \frac{C_\alpha}{n^{\alpha}}.$$
We now fix any $\beta>0$ and fix some $0<\alpha<1$. We can use the fact that we have (by the lower bound on $\lambda_*(X)$) 
$$\vert \log(\lambda_*(X))\vert^\beta=O((\log(n))^\beta),$$
where the implied constant is uniform with respect to $n$.
We therefore get as $n\rightarrow \infty$,
$$\mathbb{E}(\vert \log(\lambda_*(X))\vert^\beta \chi_0)\leq \vert \log(r_\alpha)\vert^\beta(1+O(n^{-\alpha}))+O\left(n^{-\alpha} \vert \log(n)\vert^\beta\right)=O(1)$$
and the proof is done. 
\end{proof}

We conclude with some comments.

\begin{itemize}
\item It would be interesting to know if a similar type of result can be proved for the Brooks--Makover model, in particular can one show that
$$\mathbb{E}\left ( \frac{1}{\ell_0(S_n^C)} \right)$$
is finite and uniformly bounded with respect to $n$ ? This will require an effective version of the compactification
procedure, see the paper of Mangoubi \cite{Mangoubi}.
\item It is likely that our results can be extended to finite area surfaces, see Efrat \cite{Efrat} for the definition and properties of determinants in this context, or even geometrically finite surfaces where a notion of determinant also holds, see \cite{BJP}. 
\end{itemize}


\begin{thebibliography}{10}
\bibitem{AM}
Nalini Anantharaman and Laura Monk.
\newblock Friedman-Ramanujan functions in random hyperbolic geometry and application to spectral gaps.
\newblock { \em Preprint 2023}.

\bibitem{B1}
Nicolas Bergeron, Wolfgang L\"{u}ck, and Roman Sauer.
\newblock The asymptotic growth of twisted torsion.
\newblock {\em Unpublished notes}.

\bibitem{BV}
Nicolas Bergeron and Akshay Venkatesh.
\newblock The asymptotic growth of torsion homology for arithmetic groups.
\newblock {\em J. Inst. Math. Jussieu}, 12(2):391--447, 2013.

\bibitem{Boll1}
B\'{e}la Bollob\'{a}s.
\newblock {\em Random graphs}, volume~73 of {\em Cambridge Studies in Advanced
  Mathematics}.
\newblock Cambridge University Press, Cambridge, second edition, 2001.

\bibitem{BJP}
David Borthwick, Chris Judge and Peter Perry. 
\newblock{Determinants of Laplacians and isopolar metrics on surfaces of infinite area.}
\newblock{\em Duke Math. J.}, 118 (2003), no. 1, 61--102.

\bibitem{BS}
J.~Bolte and F.~Steiner.
\newblock Determinants of {L}aplace-like operators on {R}iemann surfaces.
\newblock {\em Comm. Math. Phys.}, 130(3):581--597, 1990.

\bibitem{Brooks}
Robert Brooks.
\newblock The spectral geometry of a tower of coverings.
\newblock {\em  J. Differential Geom. }, 23 (1986) no. 1, 97--107.


\bibitem{BM2}
Robert Brooks and Eran Makover.
\newblock Belyi surfaces.
\newblock In {\em Entire functions in modern analysis ({T}el-{A}viv, 1997)},
  volume~15 of {\em Israel Math. Conf. Proc.}, pages 37--46. Bar-Ilan Univ.,
  Ramat Gan, 2001.

\bibitem{BM}
Robert Brooks and Eran Makover.
\newblock Random construction of {R}iemann surfaces.
\newblock {\em J. Differential Geom.}, 68(1):121--157, 2004.

\bibitem{Buser}
Peter Buser.
\newblock {\em Geometry and spectra of compact {R}iemann surfaces}.
\newblock Modern Birkh\"{a}user Classics. Birkh\"{a}user Boston, Ltd., Boston,
  MA, 2010.
\newblock Reprint of the 1992 edition.

\bibitem{Chavel}
Isaac Chavel.
\newblock{\em Eigenvalues in Riemannian Geometry}.
\newblock Pure Appl. Math. 115, Academic Press, 1984.




\bibitem{Cheeger}
Jeff Cheeger.
\newblock Analytic torsion and the heat equation.
\newblock {\em Ann. of Math. (2)}, 109(2):259--322, 1979.

\bibitem{DP}
Eric D'Hoker and D.~H. Phong.
\newblock On determinants of {L}aplacians on {R}iemann surfaces.
\newblock {\em Comm. Math. Phys.}, 104(4):537--545, 1986.

\bibitem{DP2}
Eric D'Hoker and D.~H. Phong.
\newblock The geometry of string perturbation theory.
\newblock {\em Rev. Modern Phys.}, 60(4):917--1065, 1988.

\bibitem{Efrat}
Isaac Efrat.
\newblock{Determinants of Laplacians on surfaces of finite volume}.
\newblock{\em Comm. Math. Phys.} 119 (1988), no. 3, 434--451.

\bibitem{PF1}
Ksenia Fedosova and Anke Pohl.
\newblock Meromorphic continuation of Selberg zeta functions with twists having non-expanding cusp monodromy.
\newblock{\em Selecta Math. (N.S.)} 26 (2020), no. 1, Paper No. 9, 55 pp. 

\bibitem{Hejhal1}
Dennis~A. Hejhal.
\newblock {\em The {S}elberg trace formula for {${\rm PSL}(2,R)$}. {V}ol. {I}}.
\newblock Lecture Notes in Mathematics, Vol. 548. Springer-Verlag, Berlin-New
  York, 1976.

\bibitem{JN}
Dmitry Jakobson and Fr\'{e}d\'{e}ric Naud.
\newblock On the critical line of convex co-compact hyperbolic surfaces.
\newblock {\em Geom. Funct. Anal.}, 22(2):352--368, 2012.

\bibitem{KSV}
Mikhail~G. Katz, Mary Schaps, and Uzi Vishne.
\newblock Logarithmic growth of systole of arithmetic {R}iemann surfaces along
  congruence subgroups.
\newblock {\em J. Differential Geom.}, 76(3):399--422, 2007.

\bibitem{LS1}
Martin~W. Liebeck and Aner Shalev.
\newblock Fuchsian groups, coverings of {R}iemann surfaces, subgroup growth,
  random quotients and random walks.
\newblock {\em J. Algebra}, 276(2):552--601, 2004.

\bibitem{MNP}
Michael Magee, Fr{\'e}d{\'e}ric Naud, and Doron Puder.
\newblock A random cover of a compact hyperbolic surface has relative spectral
  gap {$\frac{3}{16}-\epsilon$}.
\newblock {\em  Geom. Funct. Anal.}, 32 (2022), no. 3, 595--661

\bibitem{MP1}
Michael Magee and Doron Puder.
\newblock The asymptotic statistics of random covering surfaces.
\newblock {\em Preprint}, 2020.

\bibitem{Mangoubi}
Dan Mangoubi.
\newblock Conformal extension of metrics of negative curvature.
\newblock {\em J. Anal. Math.}, 91:193--209, 2003.

\bibitem{Mirza}
Maryam Mirzakhani.
\newblock Simple geodesics and {W}eil-{P}etersson volumes of moduli spaces of
  bordered {R}iemann surfaces.
\newblock {\em Invent. Math.}, 167(1):179--222, 2007.

\bibitem{Mir1}
Maryam Mirzakhani.
\newblock Growth of {W}eil-{P}etersson volumes and random hyperbolic surfaces
  of large genus.
\newblock {\em J. Differential Geom.}, 94(2):267--300, 2013.

\bibitem{MP}
Maryam Mirzakhani and Bram Petri.
\newblock Lengths of closed geodesics on random surfaces of large genus.
\newblock {\em Comment. Math. Helv.}, 94(4):869--889, 2019.

\bibitem{Monk1}
Laura Monk.
\newblock Benjamini-{S}chramm convergence and spectra of random hyperbolic
  surfaces of high genus.
\newblock {\em Anal. PDE}, 15(3):727--752, 2022.

\bibitem{Muller}
Werner M\"{u}ller.
\newblock Analytic torsion and {$R$}-torsion of {R}iemannian manifolds.
\newblock {\em Adv. in Math.}, 28(3):233--305, 1978.

\bibitem{N1}
Fr\'{e}d\'{e}ric Naud.
\newblock Random covers of compact surfaces and smooth linear spectral
  statistics.
\newblock {\em Preprint}, 2022.

\bibitem{OPS}
Osgood B., Phillips R. and Sarnak P.
\newblock Extremals of determinants of Laplacians.
\newblock {\em J. Funct. Anal.}, 80 (1988), 148--211.

\bibitem{OR}
Jean Pierre Otal and Eulalio Rosas.
\newblock Pour toute surface hyperbolique de genre $g$, $\lambda_{2g-2}>1/4$.
\newblock {\em Duke Math. J.}, 150, (2009), 101--115.

\bibitem{Petri1}
Bram Petri.
\newblock Random regular graphs and the systole of a random surface.
\newblock {\em J. Topol.}, 10(1):211--267, 2017.

\bibitem{PW}
Bram Petri and Alexander Walker.
\newblock Graphs of large girth and surfaces of large systole.
\newblock {\em Math. Res. Lett.}, 25(6):1937--1956, 2018.

\bibitem{PR}
Mark Pollicott and Andr\'{e}~C. Rocha.
\newblock A remarkable formula for the determinant of the {L}aplacian.
\newblock {\em Invent. Math.}, 130(2):399--414, 1997.

\bibitem{Polyakov}
A.~M. Polyakov.
\newblock Quantum geometry of bosonic strings.
\newblock {\em Phys. Lett. B}, 103(3):207--210, 1981.

\bibitem{Sarnak}
Peter Sarnak.
\newblock Determinants of {L}aplacians.
\newblock {\em Comm. Math. Phys.}, 110(1):113--120, 1987.

\bibitem{SX}
Peter Sarnak and Xiao~Xi Xue.
\newblock Bounds for multiplicities of automorphic representations.
\newblock {\em Duke Math. J.}, 64(1):207--227, 1991.

\bibitem{Str}
Alexander Strohmaier.
\newblock Computation of eigenvalues, spectral zeta functions and
  zeta-determinants on hyperbolic surfaces.
\newblock In {\em Geometric and computational spectral theory}, volume 700 of
  {\em Contemp. Math.}, pages 177--205. Amer. Math. Soc., Providence, RI, 2017.

\bibitem{Sunada}
Toshikazu Sunada.
\newblock Unitary representations of fundamental groups and the spectrum of
  twisted {L}aplacians.
\newblock {\em Topology}, 28(2):125--132, 1989.

\bibitem{Venkov}
Alexei~B. Venkov.
\newblock {\em Spectral theory of automorphic functions and its applications},
  volume~51 of {\em Mathematics and its Applications (Soviet Series)}.
\newblock Kluwer Academic Publishers Group, Dordrecht, 1990.
\newblock Translated from the Russian by N. B. Lebedinskaya.

\bibitem{Wolpert}
Scott~A. Wolpert.
\newblock Asymptotics of the spectrum and the {S}elberg zeta function on the
  space of {R}iemann surfaces.
\newblock {\em Comm. Math. Phys.}, 112(2):283--315, 1987.

\bibitem{WL}
Alex Wright and Michael Lipnowski.
\newblock Towards optimal spectral gap in large genus.
\newblock {\em Preprint}, 2021.

\bibitem{WX}
Yunhui Wu and Yuhao Xue.
\newblock Random hyperbolic surfaces of large genus have first eigenvalues
  greater than {$\frac{3}{16}-\epsilon$}.
\newblock {\em Geom. Funct. Anal.}, 32(2):340--410, 2022.

\end{thebibliography}
\end{document}